\documentclass{amsart} 
\usepackage{setting}

\numberwithin{equation}{section}

\begin{document}

\title{Cohen-Macaulay Weighted Chordal Graphs} 


\author{Shuai Wei} 
\address{University of New Mexico
Department of Mathematics and Statistics, 
1 University of New Mexico, MSC01 1115
Albuquerque, NM 87131 USA}
\email{wei6@unm.edu}

\begin{abstract}
    In this paper I give a combinatorial characterization of all the Cohen-Macaulay weighted chordal graphs. In particular, it is shown that a weighted chordal graph is Cohen-Macaulay if and only if it is unmixed.
\end{abstract}

\maketitle

\section{Introduction}
\noindent \textbf{Convention.} Throughout this paper, let $\bbN = \{1,2,\dots,\}$, $n \in \bbN$, $\bbK$ a field, and $G = (V,E)$ a (finite simple) graph with vertex set $V = V(G) = [n] = \{1,\dots,n\}$ and edge set $E = E(G)$. An edge between vertices $i$ and $j$ is denoted $ij$. \par
    Combinatorial commutative algebra is a branch of mathematics that uses combinatorics and graph theory to understand certain algebraic constructions; it also uses algebra to understand certain objects in combinatorics and graph theory. \par
    To the graph $G$ one associates the positive integer-valued function $\lambda: E \to \bbN$, producing a weighted graph $G_\lambda$. For a weighted graph $G_\lambda$ Paulsen and Sather-Wagstaff~\cite{MR3055580} introduce the weighted edge ideal $I(G_\lambda) \subseteq \bbK[X_1,\dots,X_n]$ which is generated by all monomials $X_i^{\lambda(ij)}X_j^{\lambda(ij)}$ such that $ij \in E$. In particular, if $\lambda$ is the constant function defined by $\lambda(ij) = 1$ for $ij \in E$, then $I(G_\lambda) = I(G)$, where $I(G)$ is the edge ideal associated to $G$ in $\bbK[X_1,\dots,X_n]$, given by $I(G) = (X_iX_j \mid ij \in E)$. \par 
    A weighted graph $G_\lambda$ is called Cohen-Macaulay over $\bbK$ if $\bbK[X_1,\dots,X_n]/I(G_\lambda)$ is a Cohen-Macaulay ring, and is called Cohen-Macaulay if it is Cohen-Macaulay over any field. The general problem is to classify the weighted graphs which are Cohen-Macaulay over $\bbK$. As for unweighted graphs one cannot expect a general classification theorem. Paulsen and Sather-Wagstaff~\cite{MR3055580} characterized all Cohen-Macaulay weighted $K_1$-corona of graphs and in particular all Cohen-Macaulay weighted trees. In this paper we classify all Cohen-Macaulay weighted chordal graphs following the classification of all Cohen-Macaulay chordal graphs by Herzog, Hibi, and Zheng\cite{MR2231097}. The characterization is purely graph-theoretical, and it turns out that for weighted chordal graphs the Cohen-Macaulay property is independent of the field $\bbK$. \par 
    In Section~\ref{2}, we recall some definitions and notations from~\cite{MR3822066}, \cite{MR2231097}, and \cite{MR3055580}. We also prove a lemma used in proving the sufficient condition for the Cohen-Macaulay property to hold for weighted chordal graphs. In Section~\ref{3}, we classify all Cohen-Macaulay weighted chordal graphs. Theorem~\ref{CMUnmixedEquivalenceThm} gives a sufficient condition and Theorem~\ref{CMUnmixedEquivalenceNecessaryThm} says that the sufficient condition is also a necessary condition. 

\section{Preliminaries} \label{2}

In subsequent sections, let $\lambda:E \to \bbN$ be a positive integer-valued function and $G_\lambda$ a weighted graph. 

\begin{definition} \cite{MR3822066}
    A \emph{path} in $G$ is a non-empty subgraph $P = (V',E')$ of the form $V' = \{v_0,v_1,\dots,v_r\}$ and $E' = \{v_0v_1,v_1v_2,\dots,v_{r-1}v_r\}$, where $r \in \bbN \sqcup \{0\}$, we denote the path by $P = v_0v_1 \cdots v_r$ for simplicity and define $\mathring P_0 = v_1 \cdots v_r$.
\end{definition}

\begin{definition} \cite{MR3822066}
    If $G' = (V',E')$ is a subgraph of $G$ and $E' = \{ij \in E \mid i,j \in V'\}$, then $G'$ is an \emph{induced subgraph} of $G$; we say that $V'$ \emph{induces} or \emph{spans} $G'$ in $G$, and write $G' = G[V']$. A subgraph $G' = (V',E')$ of $G$ is a \emph{spanning} subgraph of $G$ if $V'$ spans all of $G$, i.e., if $V' = V$. \par
    If $U \subseteq V$, we write $G - U$ for $G[V \smallsetminus U]$. \par
    A maximal connected subgraph of $G$ is a \emph{component} of $G$. If $G_1,\dots,G_n$ are components of $G$, then $G_i = G[V(G_i)]$ for $i = 1,\dots,n$ and $V(G_1),\dots,V(G_n)$ partition $V$. 
\end{definition}

\begin{definition}
    A \emph{rooted tree} $T$ is a tree with a special vertex $r \in V(T)$ labelled as the ``root'' of the tree. A \emph{rooted forest} is a graph whose components are rooted trees. A subgraph $G' = (V',E')$ of $G$ is called a \emph{rooted spanning forest} of $G$ if $G'$ is a rooted forest and $V'$ spans $G$.
\end{definition}

\begin{definition}
    Let $T$ be a rooted tree with root $r$. For any $v \in T \smallsetminus \{r\}$, there is a unique path from $r$ to $v$, say $v_1 \dots v_k$ with $v_1 = r$ and $v_k = v$, we say that $v_i$ is a \emph{$T$-parent} of $v_{i+1}$ and $v_{i+1}$ is a \emph{$T$-child} of $v_i$ for $i = 1,\dots,k-1$. \par
    If $\ell \in V(T)$ has no $T$-child, we say that $\ell$ is a \emph{$T$-leaf}, otherwise, it is a \emph{$T$-inner vertex}. In particular, if $V(T) = \{r\}$, then $r$ is a $T$-\emph{leaf} but not a $T$-inner vertex.
\end{definition}

\begin{definition} \cite{MR2231097}
    A \emph{stable subset} or \emph{clique} of the graph $G$ is a subset $F$ of $[n]$ such that $ij \in E$ for all $i,j \in F$ with $i \neq j$. We write $\Delta(G)$ for the simplicial complex on $[n]$ whose faces are the stable subsets of $G$. A vertex of $\Delta(G)$ is \emph{free} if it belongs to exactly one facet of $\Delta(G)$, otherwise it is called \emph{nonfree}.
\end{definition}

\begin{definition}
    A \emph{chord} in the graph $G$ refers to an edge that connects two non-adjacent vertices within a cycle. The graph $G$ is called \emph{chordal} if every cycle of length $>3$ has a chord. The weighted graph $G_\lambda$ is called a \emph{weighted chordal graph} if the underlying graph $G$ is chordal.
\end{definition}

\begin{definition}\cite[Definition 1.4]{MR3055580}
    A \emph{weighted vertex cover} of $G_\lambda$ is an ordered pair $(V',\delta')$ with a subset $V' \subseteq V$ and a function $\delta': V' \to \bbN$ such that for each edge $ ij \in E$, we have that
    \begin{enumerate}
        \item $i \in V'$ and $\delta'(i) \leq \lambda(ij)$, or 
        \item $j \in V'$ and $\delta'(j) \leq \lambda(ij)$.
    \end{enumerate}
    The number $\delta'(i)$ is the \emph{weight} of $v_i$. 
\end{definition}

\begin{note}
    If $(V',\delta')$ is a weighted vector cover of $G_\lambda$, then $V'$ is a vector cover of $G$ by definition.
\end{note}

\begin{definition}\cite[Definition 1.9]{MR3055580}
    Given two weighted vertex covers $(V_1',\delta_1')$ and $(V_2',\delta_2')$ of $G_\lambda$, we write $(V_2',\delta_2') \leq (V_1',\delta_1')$ if $V_2' \subseteq V_1'$ and $\delta_2' \geq \delta_1'|_{V_2'}$. A weighted vertex cover $(V',\delta')$ is \emph{minimal} if there does not exist another weighted vertex cover $(V'',\delta'')$ such that $(V'',\delta'') < (V',\delta')$. The \emph{cardinality} of $(V',\delta')$ is defined to be the cardinality of $V'$, in symbols, $\abs{(V',\delta')} = \abs{V'}$. \par
    A weighted graph $G_\lambda$ is called \emph{unmixed} if all of the minimal weighted vertex covers of $G_\lambda$ have the same cardinality, otherwise it is called \emph{mixed}.  
\end{definition}

For the proof of our first main theorem we need the following algebraic fact:

\begin{lemma} \label{NZDOnSModuloJLemma}
    Let $R$ be a Noetherian ring, $S = R[X_1,\dots,X_n]$, $k \in \{0,\dots,n-1\}$ and $J = (I_1X_1,\dots,I_kX_k,\{X_i^{m_{ij}}X_j^{m_{ij}}\}_{1 \leq i < j \leq n})$ an ideal of $S$, where $I_1,\dots,I_k$ are ideals of $R$, no $I_j$'s exists in $J$ if $k = 0$, and $m_{ij} \in \bbN$ for any $i,j \in \bbN$. Then $X = \sum_{i=1}^n X_i$ is a non-zero divisor on $S/J$.
\end{lemma}

\begin{proof}
    Let $A \subseteq [n]$ be nonempty. Let $(K_A)_{\lambda_A}$ be the weighted complete graph on $A$ with the function $\lambda_A: E(K_A) \to \bbN$ given by $\lambda_A(ij) = m_{ij}$ for $i,j \in A$ such that $1 \leq i < j \leq n$. For any weighted vertex cover $(V',\delta')$ of $(K_A)_{\lambda_A}$, define an ideal $P^A(V',\delta') := (X_i^{\delta'(i)} \mid i \in V')$ in $\bbK[X_j \mid j \in A]$. Let $I((K_A)_{\lambda_A})$ be the weighted edge ideal of $(K_A)_{\lambda_A}$ in $\bbK[X_j \mid j \in A]$. Then by~\cite[Theorem 3.5]{MR3055580}, we have that
    \[I((K_A)_{\lambda_A}) = \bigcap_{\min.\ (V',\delta')}P^A(V',\delta'),\]
    where the intersection is taken over all minimal weighted vertex covers of $(K_A)_{\lambda_A}$. For any $B \subseteq [n]$, let $\ffX_B = (X_j \mid j \in B)S$. For any $T \subseteq [k]$, set $I_T = \sum_{j \in T} I_j$. Then 
    \begin{align*}
        J &= (I_1X_1,\dots,I_kX_k,I((K_{[n]})_{\lambda_{[n]}}))S \\
        &= \bigcap_{T \subseteq [k]} (I_T, \ffX_{[k] \smallsetminus T}) + \bigcap_{(W',\gamma')} P^{[n]}(W',\gamma') \\
        &= \bigcap_{T \subseteq [k]} \bigcap_{(W',\gamma')} \bigl(I_T,\ffX_{[k] \smallsetminus T},P^{[n]}(W',\gamma')\bigr)S \\
        &= \bigcap_{T \subseteq [k]} \bigcap_{(V',\delta')} \bigl(I_T,\ffX_{[k] \smallsetminus T},P^{[n] \smallsetminus ([k] \smallsetminus T)}(V',\delta')\bigr)S,
    \end{align*}
    where $(W,\gamma')$ runs through all minimal weighted vertex covers of $(K_{[n]})_{\lambda_{[n]}}$, and $(V',\delta')$ runs through all minimal weighted vertex covers of $(K_{[n] \smallsetminus ([k] \smallsetminus T)})_{\lambda_{[n] \smallsetminus ([k] \smallsetminus T)}}$. The third equality follows from~\cite[Lemma 7.3.2]{MR3839602} since $R$ is Noetherian. \par 
    To prove that $X$ is a non-zero divisor modulo $J$ it suffices to show that $X$ is a non-zero divisor modulo each of the ideals $(I_T,\ffX_{[k] \smallsetminus T},P^{[n] \smallsetminus ([k] \smallsetminus T)}(V',\delta'))S$. It is equivalent to show that $X$ is a non-zero divisor on
    \[\frac{\overbar{R}[X_1,\dots,X_n]}{(\ffX_{[k] \smallsetminus T},P^{[n] \smallsetminus ([k] \smallsetminus T)}(V',\delta'))},\]
    where $\overbar{R} = \frac{R}{I_T}$. The associated prime of the primary ideal $(\ffX_{[k] \smallsetminus T},P^{[n] \smallsetminus ([k] \smallsetminus T)}(V',\delta'))$ generated by pure powers is $(\ffX_{[k] \smallsetminus T},\ffX_{V'})$ in $\overbar{R}[X_1,\dots,X_n]$. Since $(V',\delta')$ is a minimal weighted vertex cover of the weighted complete graph $(K_{[n] \smallsetminus ([k] \smallsetminus T)})_{\lambda_{[n] \smallsetminus ([k] \smallsetminus T)}}$, we have by \cite[Proposition 4.6]{MR3055580} there exists an $\ell \in [n] \smallsetminus ([k] \smallsetminus T)$ but $\ell \not\in V'$. Thus, $X$ is a non-zero divisor on it. 
\end{proof}

The following example illustrates the decomposition in previous lemma.

\begin{example}
    Let $R = \bbK[Y]$ be a polynomial ring and $S = R[X_1,X_2,X_3]$. Consider the ideal $J = (YX_1,YX_2,X_1^2X_2^2,X_1^2X_3^2,X_2^2X_3^2)S$ of $S$. Then with $K = (X_1^2X_2^2,X_1^2X_3^2,X_2^2X_3^2)S$,
    \begin{align*}
        J & = (YX_1,YX_2)S + K \\
        &= (X_1,X_2)S \cap (Y,X_2)S \cap (Y,X_1)S \cap (Y,Y)S + K \\
        &= (X_1,X_2,K)S \cap (Y,X_2,K)S \cap (Y,X_1,K)S \cap (Y,K)S \\
        &= (X_1,X_2)S \cap (Y,X_2,X_1^2X_3^2)S \cap (Y,X_1,X_2^2X_3^2)S \cap (Y,K)S \\
        &= (X_1,X_2)S \cap (Y,K)S \\
        &= (X_1,X_2,K)S \cap (Y,K)S \\
        &= (YX_1,YX_2)S + K.
    \end{align*}   
\end{example}

In this paper, we will use Lemma~\ref{NZDOnSModuloJLemma} in the context of $R = \bbK[X_1,\dots,X_n]$. Before proving the two main theorems, let's look at two particular examples. 

\begin{example}
    The weighted edge ideal of the following weighted graph $G_\lambda$ is mixed.
    \[
        \begin{tikzcd}
            & v_3 \ar[ld,dash,red,"2"'] \ar[ld,dash,red] \ar[ld,dash,red] \ar[ld,dash,red] \ar[ld,dash,red] \ar[ld,dash,red] \ar[ld,dash,red] \ar[rd,dash,"a"] & &[+15pt] \\
            v_2 \ar[rr,dash,"4",red] \ar[rr,dash,red] \ar[rr,dash,red] \ar[rr,dash,red] \ar[rr,dash,red] & & \nu_1 \ar[r,dash,"6"] & \nu_0 
        \end{tikzcd}
    \]
    Theorem~\cite[Theorem 3.5]{MR3055580} says that it suffices to find two minimal weighted cover of $G_\lambda$ of different cardinality. Note that there always exists a minimal weighted vertex cover of size $3 + 1 - 2 = 2$. For example, $(V',\delta') = \{\nu_1^1,v_2^1\}$ is a weighted vertex cover of $G_\lambda$, and it is cardinality-minimal in the sense that there doesn't exist any weighted vertex cover $(W',\gamma')$ of $G_\lambda$ such that $(W',\gamma') \leq (V,\delta')$ and $\abs{W'} < \abs{V'}$. By~\cite[Proposition 1.12]{MR3055580}, $(V',\delta')$ induces a minimal weighted vertex cover $(V',\delta'')$ for some $\delta'' \geq \delta'$. But there exists another minimal weighted vertex cover $\{v_3^{\min\{2,a\}},v_2^4,\nu_1^6\}$, whose size is 3. 
\end{example}

\begin{example}
    The weighted edge ideal of the following weighted graph $G_\lambda$ is mixed.
    
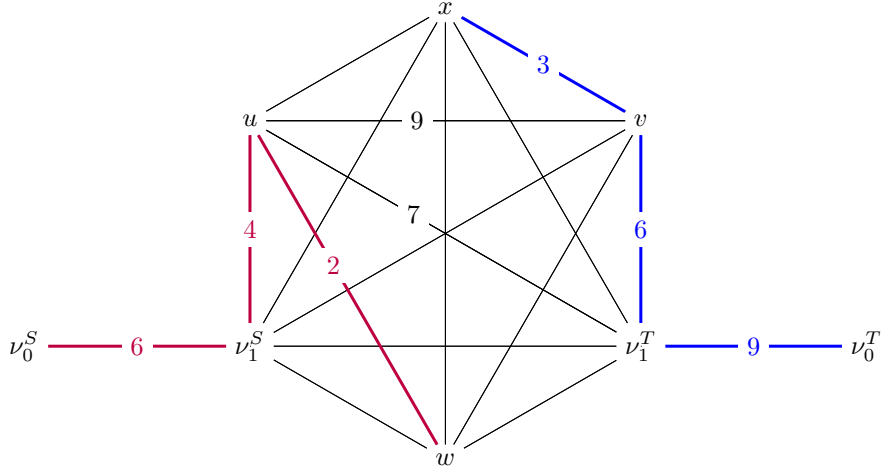
\[
        \begin{tikzpicture}[line width=0.0mm]
            \foreach \i\j in {0/x,60/v,180/w,300/u}{
                \node (\j) at ({3*sin(\i)}, {3*cos(\i)}) {$\j$};
            }
            \node (c) at ({3*sin(120)}, {3*cos(120)}) {$\nu_1^T$}; 
            \node (e) at ({3*sin(240)}, {3*cos(240)}) {$\nu_1^S$};
            \foreach \i in {x,v,c,w,e,u}{ 
                \foreach \j in {x,v,c,w,e,u}{ 
                    \draw (\i)--(\j);
                }
            }
            \node (g) at ({-3+3*sin(240)}, {3*cos(240)}) {$\nu_0^S$}; 
            \node (h) at ({3+3*sin(120)}, {3*cos(120)}) {$\nu_0^T$};

            \draw[purple,line width = 0.4mm] (e)--(u) node [midway,fill=white] {4};  
            \draw[purple,line width = 0.4mm] (g)--(e) node [midway,fill=white] {6}; 
            \draw[purple,line width = 0.4mm] (u)--(w) node [pos=0.42,fill=white] {2}; 
     
            \draw[blue,line width = 0.4mm] (h)--(c) node [midway,fill=white] {9}; 
            \draw[blue,line width = 0.4mm] (c)--(v) node [midway,fill=white] {6}; 
            \draw[blue,line width = 0.4mm] (v)--(x) node [midway,fill=white] {3}; 
            \draw[line width = 0.1mm] (u)--(v) node [pos=0.42,fill=white] {9}; 
            \draw[line width = 0.1mm] (u)--(c) node [pos=0.42,fill=white] {7};
        \end{tikzpicture}  
    \]
    There exists a minimal weighted vertex cover of $G_\lambda$ of size $6+1+1-3 = 5$. Note that $(V',\delta') := \{(\nu_1^S)^{6},u^4,w^1,(\nu_1^T)^{9},v^6,x^1\}$ is a weighted vertex cover of $G_\lambda$ of cardinality 6, which is cardinality-minimal. Thus $(V',\delta')$ induces a minimal weighted vertex cover $(V',\delta'')$ for some $\delta'' \geq \delta'$.
\end{example}

\section{The conditions of a weighted chordal graph to be Cohen-Macaulay}  \label{3}

The next result gives a sufficient condition for weighted chordal graphs to be Cohen-Macaulay.

\begin{theorem} \label{CMUnmixedEquivalenceThm}
    Let $G_\lambda$ be a weighted chordal graph. Let $F_1,\dots,F_m$ be the facets of $\Delta(G)$ which admit a free vertex satisfying that for $i = 1,\dots,m$: there doesn't exist a rooted spanning forest $\ffF$ of $G[F_i]$ in which each component $T$ has a nonfree vertex $\nu_1^T$ as the root such that there is a nonfree vertex $\nu_0^T$ in $V \smallsetminus F_i$ with $\nu_0^T\nu_1^T \in E$ which satisfies that for any path $\nu_1^Tv_2 \cdots v_k$ in $T$ with $v_k$ a $T$-leaf: if $k \geq 2$, then $\lambda(\nu_0^T\nu_1^T) > \lambda(\nu_1^Tv_2) > \lambda(v_2v_3) > \cdots > \lambda(v_{k-1}v_k)$, and which satisfies that for any component(s) $S,T$ of $\ffF$: if $u$ is an $S$-inner vertex and $v$ is a $T$-inner vertex with $u \neq v$, then 
    \[\lambda(uv) > \min\{\max\{\lambda(uw) \mid \text{$u$ is the $S$-parent of $w$}\},\max\{\lambda(vx) \mid \text{$v$ is the $T$-parent of $x$}\}\};\]
    if $u$ is an $S$-inner vertex such that $u \nu_0^T \in E$, then
    \begin{align*}
        \lambda(u\nu_0^T) &> \min\{\max\{\lambda(uw) \mid \text{$u$ is the $S$-parent of $w$}\}, \\
        &\ \ \ \ \ \ \ \ \ \, \max\{\lambda(\nu_0^T\nu_1^Y) \mid \text{$Y$ is a component of $\ffF$ such that $\nu_0^T = \nu_0^Y$}\}\};
    \end{align*}
    and if $\nu_0^S\nu_0^T \in E$, then 
    \begin{align*}
        \lambda(\nu_0^S\nu_0^T) &> \min\{\max\{\lambda(\nu_0^S\nu_1^Y) \mid \text{$Y$ is a component of $\ffF$ such that $\nu_0^S = \nu_0^Y$}\}, \\
        &\ \ \ \ \ \ \ \ \ \, \max\{\lambda(\nu_0^T\nu_1^Z) \mid \text{$Z$ is a component of $\ffF$ such that $\nu_0^T = \nu_0^Z$}\}\}.
    \end{align*}
    Then the following conditions are equivalent.
    \begin{enumerate} 
        \item $G_\lambda$ is Cohen-Macaulay;
        \item $G_\lambda$ is Cohen-Macaulay over $\bbK$;
        \item $G_\lambda$ is unmixed;
        \item $[n]$ is the disjoint union of $F_1,\dots,F_m$.
    \end{enumerate}
\end{theorem}

\begin{proof}
    (a) $\Longrightarrow$ (b) is trivial. \par
    (b) $\Longrightarrow$ (c) Since $\bbK[X_1,\dots,X_n]/I(G_\lambda)$ is Cohen-Macaulay, we have $I(G_\lambda)$ is unmixed. So $G_\lambda$ is unmixed by~\cite[Theorem 3.5 and Proposition 3.13]{MR3055580}. \par
    (c) $\Longrightarrow$ (d) Since $G_\lambda$ unmixed, we have $G$ is unmixed by~\cite[Proposition 1.14]{MR3055580}. So (d) holds by~\cite[Theorem 2.1]{MR2231097}. \par
    (d) $\Longrightarrow$ (c) Let $(V',\delta')$ be a minimal weighted vertex cover of $G_\lambda$ with $V' \subseteq [n]$ and $\delta': V' \to \bbN$. Then for $i = 1,\dots,m$ we have $\abs{V' \cap F_i} \geq \abs{F_i} -1$ since $F_i$ is a clique of $G$. Suppose for some $i \in \{1,\dots,m\}$, we have $\abs{V' \cap F_i} = \abs{F_i}$, i.e., $F_i \subseteq V'$. Then $F_i$ contains a nonfree vertex by~\cite[Proposition 4.6]{MR3055580}. Let $v \in F_i$. Claim. There exists a path $\ell_0\ell_1 \cdots \ell_{k-1}\ell_k$ in $G$ with $\ell_0 \not\in F_i$, $\ell_1,\dots,\ell_{k-1} \in F_i$ and $\ell_k = v$ satisfying that if $k \geq 2$, then $\lambda(\ell_0\ell_1) > \cdots > \lambda(\ell_{k-2}\ell_{k-1}) > \lambda(\ell_{k-1}\ell_k)$. \par
    We will use an algorithm to find such a path. Since $v \in F_i \subseteq V'$, we have there exists $w_1 \in V \smallsetminus \{v\}$ such that $\delta'(v) \leq \lambda(vw_1) \smallunderbrace{< \delta'(w_1)}_{\text{if }w_1 \in V'}$. We then go through the following steps: 
    \begin{enumerate}
        \item [Step 1.]
            Initially, let $j := 1$.
        \item [Step 2.]
            If $w_j \not\in F_i$ and $j = 1$, then we have a path $w_1v$ in $G$ with $w_1 \not \in F_j$ and $v \in F_i$, so the claim is justified.  If $w_j \not\in F_i$ and $j \geq 2$, then by induction we have a path $w_jw_{j-1} \cdots w_1v$ in $G$ with $w_j \not\in F_i$ and $w_{j-1},\dots,w_1,v \in F_i$ such that 
            \[\delta'(v) \leq \lambda(vw_1) < \delta'(w_1) \leq \lambda(w_1w_2) < \delta'(w_2) \leq \cdots < \delta'(w_{j-1}) \leq \lambda(w_{j-1}w_j) \smallunderbrace{< \delta'(w_j)}_{\text{if }w_j \in V'},\]
            implying that $\lambda(w_jw_{j-1}) > \lambda(w_{j-1}w_{j-2}) > \cdots > \lambda(w_2w_1) > \lambda(w_1v)$ and so the claim is justified. Hence in either case we jump out of the loop. 
        \item [Step 3.]
            If $w_j \in F_i$, then there exists $w_{j+1} \in V \smallsetminus \{w_j\}$ such that $\delta'(w_j) \leq \lambda(w_jw_{j+1}) \smallunderbrace{< \delta'(w_{j+1})}_{\text{if }w_{j+1} \in V'}$, so by induction there exists a path $w_{j+1}w_jw_{j-1} \dots w_1v$ with $w_j,\dots,w_1,v \in F_i$ such that 
            \[\delta'(v) \leq \lambda(vw_1) < \delta'(w_1) \leq \lambda(w_1w_2) < \delta'(w_2) \leq \cdots < \delta'(w_j) \leq \lambda(w_jw_{j+1}) \underbrace{< \delta'(w_{j+1})}_{\text{if }w_{j+1} \in V'}.\]
        \item [Step 4.]
            Re-define $j := j+1$. If $w_j \not\in F_i$, then go back to Step 2. If $w_j \in F_i$, then go back to Step 3.
    \end{enumerate}
    Since $\abs{F_i}$ is finite and $w_j,\dots,w_1,v \in F_i$ are distinct to each other and $F_i$ contains a nonfree vertex, we have after some finite loops, it will enter Step 2 and the claim will be proved. \par
    Let $v_1 \in F_i$. Then by the claim there exists a path $P_1 := \ell_0\ell_1 \cdots \ell_{k-1}v_1$ in $G$ with $\ell_0 \not\in F_i$ and $\ell_1,\dots,\ell_{k-1},v_1 \in F_i$ satisfying that if $k \geq 2$, then $\lambda(\ell_0\ell_1) > \cdots > \lambda(\ell_{k-2}\ell_{k-1}) > \lambda(\ell_{k-1}v_1)$. Assume $V(\mathring P_1) = F_i$. Then $k \geq 2$ since $F_i$ contains a free vertex. So there exists a rooted spanning forest $\mathring P_1$ such that the unique component $\mathring P_1$ has a nonfree vertex $\nu_{1}^{\mathring P_1} := \ell_1$ as the root and there is a nonfree vertex $\nu_0^{\mathring P_1} := \ell_0$ in $V \smallsetminus F_i$ with $\ell_0\ell_1 \in E$ which satisfies that for the unique path $\ell_1 \cdots \ell_{k-1}v_1$ in the tree $\mathring P_1$ with $v_1$ a $\mathring P_1$-leaf if $k \geq 2$, then $\lambda(\ell_0\ell_1) > \cdots > \lambda(\ell_{k-2}\ell_{k-1}) > \lambda(\ell_{k-1}v_1)$. Moreover, if $\ell_\alpha$ and $\ell_\beta$ are $\mathring P_1$-inner vertices with $1 \leq \alpha < \beta \leq k-1$, then since $(V',\delta')$ is a weighted vertex cover of $G_\lambda$ and $\ell_\alpha,\ell_\beta \in V(\mathring P_1) = F_i \subseteq V'$, after settting $\ell_k = v_1$ we have
    \begin{align*}
        \lambda(\ell_\alpha\ell_\beta) &\geq \min\{\delta'(\ell_\alpha),\delta'(\ell_\beta)\} \\
        & > \min\left\{\lambda(\ell_\alpha \ell_{\alpha+1}),\lambda(\ell_\beta \ell_{\beta+1})\right\} \\
        & = \min\{ \max\{\lambda(\ell_\alpha w) \mid \text{$\ell_\alpha$ is the $\mathring P_1$-parent of $w$}\},\max\{\lambda(\ell_\beta x) \mid \text{$\ell_\beta$ is the $\mathring P_1$-parent of $x$}\}\};
    \end{align*}
    and if $\ell_\alpha$ is an $\mathring P_1$-inner vertex such that $u \nu_0^{\mathring P_1} \in E$, then $\ell_\alpha = \ell_1$, and so
    \begin{align*}
        \lambda(\ell_\alpha\nu_0^{\mathring P_1}) &= \lambda(\ell_0\ell_1) \\
        &> \max\{\lambda(\ell_1w) \mid \text{$\ell_1$ is the $\mathring P_1$-parent of $w$}\} \\
        &\geq \min\{\max\{\lambda(\ell_\alpha w) \mid \text{$\ell_\alpha$ is the $\mathring P_1$-parent of $w$}\},\lambda(\nu_0^{\mathring P_1}\nu_1^{\mathring P_1})\},
    \end{align*}
    a contradiction. \par
    On the other hand, we assume $V(\mathring P_1) \subsetneq F_i$. We then go through the following steps: 
    \begin{enumerate}
        \item [Step 1.]
            Initially, let $b := 1$.
        \item [Step 2.]
            If there exists a vertex $v \in F_i \smallsetminus (V(\mathring P_1) \cup \cdots \cup V(\mathring P_b))$, then there exists a path $P_{b+1} := h_0h_1 \cdots h_{k'-1}v$ in $G$ with $h_0 \not\in F_i$ and $h_1,\dots,h_{k'-1},v \in F_i$ satisfying that if $k' \geq 2$, then $\lambda(h_0h_1) > \cdots > \lambda(h_{k'-2}\lambda_{k'-1}) > \lambda(h_{k'-1}v)$. If $V(\mathring P_{b+1}) \cap (V(\mathring P_1) \cup \cdots \cup V(\mathring P_b)) = \emptyset$, then we put $\mathring P_{b+1}$ into the rooted forest formed by $\mathring P_1,\dots,\mathring P_b$ while making $\mathring P_{b+1}$ a rooted tree with root $h_1$ and making $\mathring P_{b+1}$ a component of the rooted forest. Assume $V(\mathring P_{b+1}) \cap (V(\mathring P_1) \cup \cdots \cup V(\mathring P_b)) \neq \emptyset$. Let $t = \max\{c \geq 1 \mid h_c \in V(\mathring P_1) \cup \cdots \cup V(\mathring P_b)\}$. Assume $h_t = u_s \in V(\mathring P_d)$ with $P_d := u_0u_1 \cdots u_j$ for some $d \in \{1,\dots,b\}$ and $s \in \{1,\dots,j\}$. Set $P_{b+1}' := u_0u_1 \cdots u_{s-1}h_t h_{t+1} \cdots h_{k'-1}v$. Then $P_{b+1}'$ is a path in $G$ with $u_0 \not \in F_i$ and $u_1,\dots,u_{s-1},h_t,h_{t+1},\dots,h_{k'-1},v \in F_i$.  Since $\lambda(u_{s-1}h_t) = \lambda(u_{s-1}u_s) \geq \delta'(u_s) = \delta'(h_t) > \lambda(h_th_{t+1})$, we have $\lambda(u_0u_1) > \cdots > \lambda(u_{s-1}h_t) > \lambda(h_th_{t+1}) > \cdots > \lambda(h_{k'-1}v)$. Re-define $P_{b+1} := P_{b+1}'$, then $\mathring P_{b+1}$ is merged into the component whose root is $u_1$, in the rooted forest formed by $\mathring P_1,\dots,\mathring P_b$.
        \item [Step 3.]
            If $V(\mathring P_1) \cup \cdots \cup V(\mathring P_{b+1}) \subsetneq F_i$, then re-define $b := b+1$ and go back to Step 2.
        \item [Step 4.]
            If $V(\mathring P_1) \cup \cdots \cup V(\mathring P_{b+1}) = F_i$, then there exists a rooted spanning forest $\ffF$ such that by induction each component $T$ in $\ffF$, which is formed by a subset of the paths $\mathring P_1,\dots,\mathring P_{b+1}$ say $\mathring P_{i_1},\dots,\mathring P_{i_\iota}$, has the nonfree vertex $\nu_1^T$ as the root, which is the first vertex of any one of the paths $\mathring P_{i_1},\dots,\mathring P_{i_\iota}$, such that there is a nonfree vertex $\nu_0^T$ in $V \smallsetminus F_i$ with $\nu_0^T\nu_1^T \in E$, which is the first vertex of any one of the paths $P_{i_1},\dots,P_{i_\iota}$, which satisfies that for the each path $\nu_1^Tu_2 \cdots u_{j-1}u_j$ in the tree $T$ with $u_j$ a $T$-leaf if $j \geq 2$, then $\lambda(\nu_0^T\nu_1^T) > \lambda(\nu_1^Tu_2) > \cdots > \lambda(u_{j-1}u_j)$. Moreover, for any component(s) $S,T$ of $\ffF$, if $u$ is an $S$-inner vertex and $v$ is a $T$-inner vertex with $u \neq v$, then since $(V',\delta')$ is a weighted vertex cover of $G_\lambda$ and $u,v \in  (V(\mathring P_1) \cup \cdots \cup V(\mathring P_{b+1})) = F_i \subseteq V'$, we have 
            \begin{align*}
                \lambda(uv) &\geq \min\{\delta'(u),\delta'(v)\} \\
                &> \min\{\max_{1 \leq c \leq {b+1}}\{\lambda(uw) \mid uw \in E(\mathring P_c)\},\max_{1 \leq d \leq {b+1}}\{\lambda(vx) \mid vx \in E(\mathring P_d)\}\} \\
                &= \min\{\max\{\lambda(uw) \mid \text{$u$ is the $S$-parent of $w$}\},\max\{\lambda(vx) \mid \text{$v$ is the $T$-parent of $x$}\}\};
            \end{align*}
            if $u$ is an $S$-inner vertex such that $u \nu_0^T \in E$, then since $(V',\delta')$ is a weighted vertex cover of $G_\lambda$, 
            \begin{align*}
                \lambda(u\nu_0^T) &\geq \min\{\delta'(u)\smallunderbrace{, \delta'(\nu_0^T)}_{\text{if }\nu_0^T \in V'}\}\\
                &> \min\{\max_{1 \leq c \leq b+1}\{\lambda(uw) \mid uw \in E(\mathring P_c)\}, \\
                &\ \ \ \ \ \ \ \ \ \, \max\{\lambda(\nu_0^T\nu_1^Y) \mid \text{$Y$ is a component of $\ffF$ such that $\nu_0^T = \nu_0^Y$}\}\} \\
                &= \min\{\max\{\lambda(uw) \mid \text{$u$ is the $S$-parent of $w$}\}, \\
                &\ \ \ \ \ \ \ \ \ \, \max\{\lambda(\nu_0^T\nu_1^Y) \mid \text{$Y$ is a component of $\ffF$ such that $\nu_0^T = \nu_0^Y$}\}\};
            \end{align*}
            and if $\nu_0^S\nu_0^T \in E$, then since $(V',\delta')$ is a weighted vertex cover of $G_\lambda$, by symmetry we assume $\nu_0^S \in V'$ and $\delta'(\nu_0^S) \leq \lambda(\nu_0^S\nu_0^T)$, so
            \begin{align*}
                \lambda(\nu_0^S\nu_0^T) &\geq \delta'(\nu_0^S) \\
                &> \max\{\lambda(\nu_0^S\nu_1^Y) \mid \text{$Y$ is a component of $\ffF$ such that $\nu_0^S = \nu_0^Y$}\} \\
                &\geq \min\{\max\{\lambda(\nu_0^S\nu_1^Y) \mid \text{$Y$ is a component of $\ffF$ such that $\nu_0^S = \nu_0^Y$}\},\\
                &\ \ \ \ \ \ \ \ \ \, \max\{\lambda(\nu_0^T\nu_1^Z) \mid \text{$Z$ is a component of $\ffF$ such that $\nu_0^T = \nu_0^Z$}\}\},
            \end{align*}
            a contradiction.
    \end{enumerate}
    Since $\abs{F_i}$ is finite, after some finite loops it will enter Step 4 and there will be a contradiction. \par
    Thus, $\abs{V' \cap F_i} = \abs {F_i}-1$ for $i = 1,\dots,m$. Since $[n]$ is the disjoint union $[n] = \bigcup_{i=1}^mF_i$, it follows that $\abs{(V',\delta')} = \abs {V'} = n-m$ and $G_\lambda$ is unmixed. \par
    (c) and (d) $\Longrightarrow$ (a) We have that $G_\lambda$ is unmixed and each minimal weighted vertex cover of $G_\lambda$ has cardinality $n-m$. Let $S = \bbK[X_1,\dots,X_n]$. Then $\dim S/I(G_\lambda) = m$. Let $Y_i = \sum_{j \in F_i}X_j$ for $i = 1,\dots,m$. We will show that $Y_1,\dots,Y_m$ is a regular sequence on $S/I(G_\lambda)$. This then yields that $G_\lambda$ is Cohen-Macaulay. Let $i \in \{2,\dots,m\}$, $F_i := \{i_1,\dots,i_\ell\}$ and assume that $i_1,\dots,i_k$ are the nonfree vertices of $G[F_i]$. Then $k \in \{0,\dots,n-1\}$. Let $G' = G[[n] \smallsetminus \{i_1,\dots,i_\ell\}]$ and $\lambda' = \lambda|_{E(G')}$. Then $I(G_\lambda) = (I(G'_{\lambda'}),J_1X_{i_1},\dots,J_k X_{i_k},J)$, where $J_j = (X_{i_j}^{\lambda(i_jr)-1}X_{r}^{\lambda(i_jr)} \mid i_jr \in E)$ for $j = 1,\dots,k$, and $J = (X_{i_r}^{\lambda(i_ri_s)}X_{i_s}^{\lambda(i_ri_s)} \mid 1 \leq r < s \leq \ell)$. Since $[n]$ is the disjoint union of $F_1,\dots,F_m$ we have all generators of the ideal $(I(G'_{\lambda'}),Y_1,\dots,Y_{i-1})$ belong to $\bbK[X_j \mid j \in [n] \smallsetminus F_i]$. Thus, if we set $R = \frac{\bbK[X_j \mid j \in [n] \smallsetminus F_i]}{(I(G'_{\lambda'}),Y_1,\dots,Y_{i-1})}$,
    then
    \[\frac{\frac{S}{I(G_\lambda)}}{(Y_1,\dots,Y_{i-1})\frac{S}{I(G_\lambda)}} \cong \frac{R[X_{i_1},\dots,X_{i_\ell}]}{(I_1X_{i_1},\dots,I_k X_{i_k},\{X_{i_r}^{\lambda(i_ri_s)}X_{i_s}^{\lambda(i_ri_s)} \mid 1 \leq r < s \leq \ell\})},\]
    where for each $j$, the ideal $I_j$ is the image of $J_j$ under the residue class map onto $R$. Thus, Lemma~\ref{NZDOnSModuloJLemma} implies that $Y_i$ is regular on $(S/I(G_\lambda))/(Y_1,\dots,Y_{i-1})(S/I(G_\lambda))$.
\end{proof}

We use the following example to illustrate the previous theorem and its proof.

\begin{example}
    The following weighted chordal graph $G_\lambda$ where we only give part weights of $\lambda$ is not Cohen-Macaulay. In the drawing, let $\nu^\tau$ denote an element of $(V',\delta')$ with $\nu \in V'$ and $\delta'(\nu) = \tau$. 
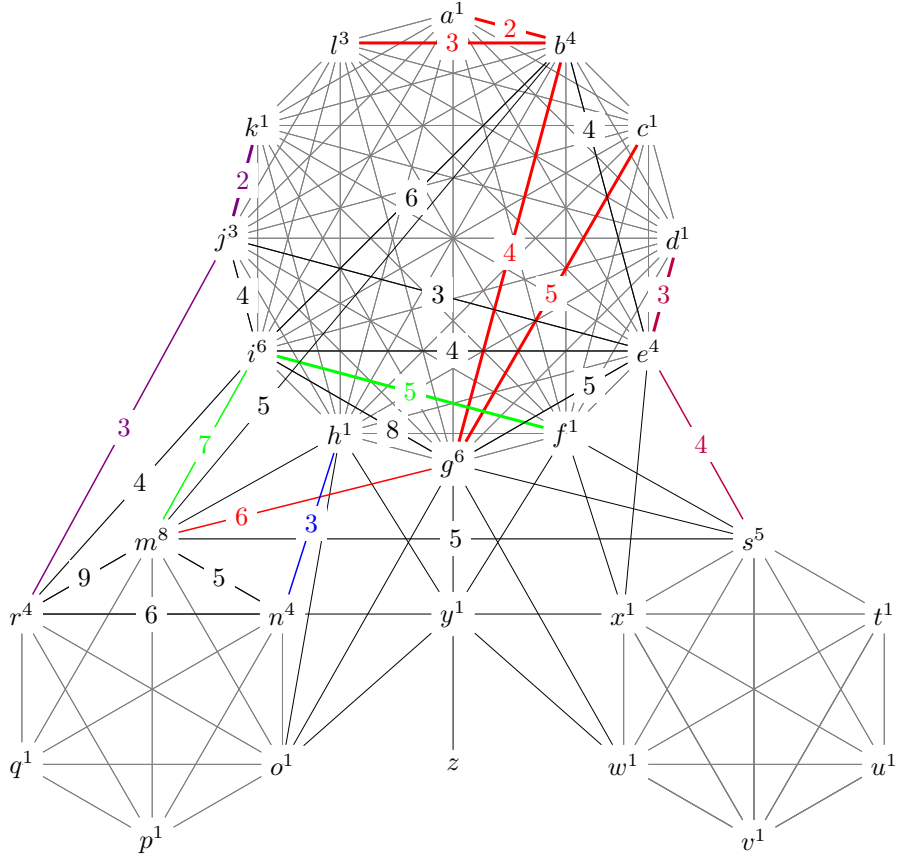
    \[
        \begin{tikzpicture}[line width=0.0mm]
            \foreach \i\j in {0/a^1,30/b^4,60/c^1,90/d^1,120/e^4,150/f^1,180/g^6,210/h^1,240/i^6,270/j^3,300/k^1,330/l^3}{
                \node (\j) at ({3*sin(\i)}, {3*cos(\i)}) {$\j$};
            }
            \foreach \i in {a^1,b^4,c^1,d^1,e^4,f^1,g^6,h^1,i^6,j^3,k^1,l^3}{ 
                \foreach \j in {a^1,b^4,c^1,d^1,e^4,f^1,g^6,h^1,i^6,j^3,k^1,l^3}{ 
                    \draw[gray] (\i)--(\j);
                }
            }   
            \foreach \i\j in {0/m^8,60/n^4,120/o^1,180/p^1,240/q^1,300/r^4}{
                \node (\j) at ({-4+2*sin(\i)}, {-6+2*cos(\i)}) {$\j$};
            }
            \foreach \i in {m^8,n^4,o^1,p^1,q^1,r^4}{ 
                \foreach \j in {m^8,n^4,o^1,p^1,q^1,r^4}{ 
                    \draw[gray] (\i)--(\j);
                }
            }   
            \foreach \i\j in {0/s^5,60/t^1,120/u^1,180/v^1,240/w^1,300/x^1}{
                \node (\j) at ({4+2*sin(\i)}, {-6+2*cos(\i)}) {$\j$};
            }
            \foreach \i in {s^5,t^1,u^1,v^1,v^1,w^1,x^1}{ 
                \foreach \j in {s^5,t^1,u^1,v^1,v^1,w^1,x^1}{ 
                    \draw[gray] (\i)--(\j);
                }
            }   
            \node (y^1) at (0, {-6+2*cos(60)}) {$y^1$};
            \node (z) at (0, {-6+2*cos(120)}) {$z$};
            \draw (y^1)--(z);
            \draw (y^1)--(f^1);
            \draw (y^1)--(h^1);
            \draw (y^1)--(n^4);
            \draw (y^1)--(x^1);
            \draw (s^5)--(f^1); 
            \draw (m^8)--(h^1); 
            \draw (s^5)--(g^6); 
            \draw (e^4)--(x^1); 
            \draw (f^1)--(x^1);
            \draw (g^6)--(y^1); 
            \draw (g^6)--(w^1); 
            \draw (g^6)--(o^1); 
            \draw (h^1)--(o^1); 
            \draw (o^1)--(y^1); 
            \draw (w^1)--(y^1); 
            \draw[black,line width = 0.15mm] (m^8)--(s^5) node [midway,fill=white] {5};
            \draw[red,line width = 0.4mm] (c^1)--(g^6) node [midway,fill=white] {5};
            \draw[red,line width = 0.4mm] (a^1)--(b^4) node [midway,fill=white] {2};
            \draw[red,line width = 0.4mm] (b^4)--(l^3) node [midway,fill=white] {3};
            \draw[red,line width = 0.4mm] (b^4)--(g^6) node [midway,fill=white] {4};
            \draw[green,line width = 0.4mm] (i^6)--(f^1) node [midway,fill=white] {5};
            \draw[green,line width = 0.2mm] (i^6)--(m^8) node [midway,fill=white] {7};
            \draw[red,line width = 0.2mm] (g^6)--(m^8) node [near end,fill=white] {6};
            \draw[violet,line width = 0.4mm] (k^1)--(j^3) node [midway,fill=white] {2};
            \draw[violet,line width = 0.2mm] (j^3)--(r^4) node [midway,fill=white] {3};
            \draw[black,line width = 0.0mm] (m^8)--(b^4) node [near start,fill=white] {5};
            \draw[blue,line width = 0.2mm] (h^1)--(n^4) node [midway,fill=white] {3};
            \draw[purple,line width = 0.4mm] (d^1)--(e^4) node [midway,fill=white] {3};
            \draw[purple,line width = 0.2mm] (e^4)--(s^5) node [midway,fill=white] {4};
            \draw[black,line width = 0.15mm] (r^4)--(i^6) node [midway,fill=white] {4};
            \draw[black,line width = 0.15mm] (m^8)--(r^4) node [midway,fill=white] {9}; 
            \draw[black,line width = 0.15mm] (m^8)--(n^4) node [midway,fill=white] {5}; 
            \draw[black,line width = 0.15mm] (r^4)--(n^4) node [midway,fill=white] {6}; 
            \draw[black,line width = 0.15mm] (i^6)--(b^4) node [midway,fill=white] {6}; 
            \draw[black,line width = 0.15mm] (i^6)--(j^3) node [midway,fill=white] {4}; 
            \draw[black,line width = 0.15mm] (j^3)--(e^4) node [midway,fill=white] {3}; 
            \draw[black,line width = 0.15mm] (b^4)--(e^4) node [near start,fill=white] {4}; 
            \draw[black,line width = 0.15mm] (g^6)--(e^4) node [near end,fill=white] {5}; 
            \draw[black,line width = 0.15mm] (g^6)--(i^6) node [near start,fill=white] {8}; 
            \draw[black,line width = 0.15mm] (i^6)--(e^4) node [midway,fill=white] {4}; 
        \end{tikzpicture}
    \]
    Let $F_1 = \{a,\dots,l\}$, $F_2 = \{m,\dots,r\}$, $F_3 = \{s,\dots,x\}$ and $F_4 = \{y,z\}$. There is a rooted spanning forest $\ffF$ of $G[F_1]$ consisting of components $T_1,\dots,T_5$ whose roots are $\nu_1^{T_1} := g,\nu_1^{T_2} := e, \nu_1^{T_3} := i, \nu_1^{T_4} = h, \nu_1^{T_5} = j$, respectively. 
    
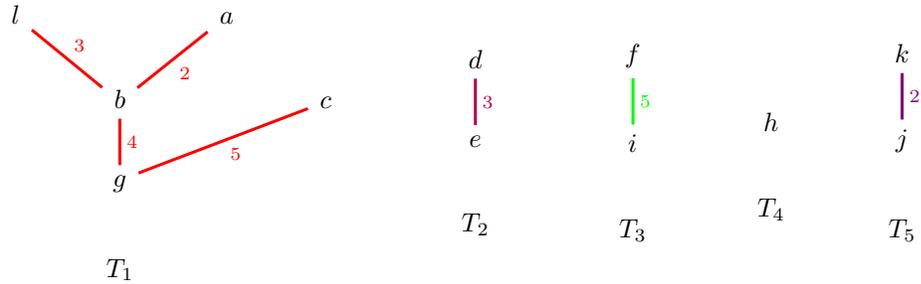
\[
        \begin{tikzcd}
            l \ar[rd,dash,red,line width = 0.4mm,"3"] && a \ar[ld,dash,red,line width = 0.4mm,"2"] \\
            & b \ar[d,dash,red,line width = 0.4mm,"4"] & & c \ar[lld,dash,red,line width = 0.4mm,"5"] \\
            & g \\
            & T_1
        \end{tikzcd} \ \ \ \ \
        \begin{tikzcd}
            & d \ar[d,dash,purple,line width = 0.4mm,"3"] \\
            & e \\
            & T_2
        \end{tikzcd} \ \ \ \ \ 
        \begin{tikzcd}
            & f \ar[d,dash,green,line width = 0.4mm,"5"] \\
            & i \\
            & T_3
        \end{tikzcd} \ \ \ \ \ \ \ \ \ \
        \begin{tikzcd}
            \\
            h \\
            T_4 
        \end{tikzcd} \ \
        \begin{tikzcd}
            & k \ar[d,dash,violet,line width = 0.4mm,"2"] \\
            & j \\
            & T_5
        \end{tikzcd}
    \]
    Let $\nu_0^{T_1} := m$, $\nu_0^{T_2} :=  s$, $\nu_0^{T_3} := m$, $\nu_0^{T_4} := n$ and $\nu_0^{T_5} := r$. For example, for the path $gbl$ in $T_1$ with $l$ a $T_1$-leaf, we have $\lambda(\nu_0^{T_1}\nu_1^{T_1}) > \lambda(\nu_1^{T_1}b) > \lambda(bl)$ since $6 > 4 > 3$.
    
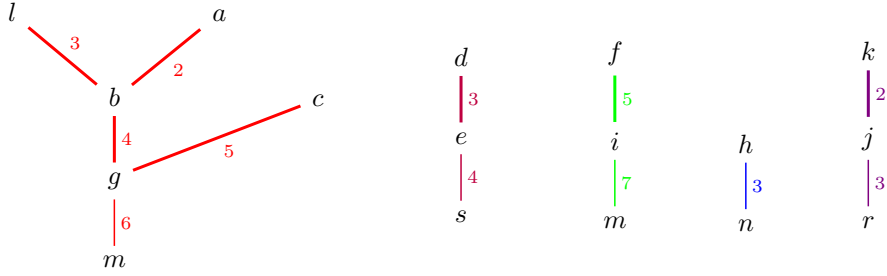
\[
        \begin{tikzcd}
            l \ar[rd,dash,red,line width = 0.4mm,"3"] && a \ar[ld,dash,red,line width = 0.4mm,"2"] \\
            & b \ar[d,dash,red,line width = 0.4mm,"4"] & & c \ar[lld,dash,red,line width = 0.4mm,"5"] \\
            & g \ar[d,dash,red,line width = 0.2mm,"6"] \\
            & m
        \end{tikzcd} \ \ \ \ \
        \begin{tikzcd}
            & d \ar[d,dash,purple,line width = 0.4mm,"3"] \\
            & e \ar[d,dash,purple,line width = 0.2mm,"4"] \\
            & s
        \end{tikzcd} \ \ \ \ \ 
        \begin{tikzcd}
            & f \ar[d,dash,green,line width = 0.4mm,"5"] \\
            & i \ar[d,dash,green,line width = 0.2mm,"7"] \\
            & m
        \end{tikzcd} \ \ \ \ \ \ \ \ \ \
        \begin{tikzcd}
            \\ \\
            h \ar[d,dash,blue,line width = 0.2mm,"3"] \\
            n
        \end{tikzcd} \ \ 
        \begin{tikzcd}
            & k \ar[d,dash,violet,line width = 0.4mm,"2"] \\
            & j \ar[d,dash,violet,line width = 0.2mm,"3"] \\
            & r
        \end{tikzcd}
    \]
    Let $\llS = \{m,n,s,r\}$, $i = 3$, $V_2 = F_2 \cap \llS = \{\nu_0^{T_1} = m = \nu_0^{T_3},\nu_0^{T_4} = n, \nu_0^{T_5} = r\}$, $V_3 = F_3 \cap S = \{\nu_0^{T_2} = s\}$. Let 
    \[W = F_2 \sqcup F_3 \sqcup (F_4 \smallsetminus \{z\}) = \{m,\dots,y\}.\]
    Let $\delta: F_1 \sqcup W \to \bbN$. The definition for $\delta$ is given in the drawing for $G_\lambda$. For example, since $m \in V_2$, 
    \begin{align*}
        \delta(m) &= \gamma_2(m)\\
        &=1 + \max_{j \in \{1,2,3,4,5\}}\{\lambda(m\nu_1^{T_j}) \mid m = \nu_0^{T_j}\} \\
        &=1 + \max_{j \in \{1,3,4,5\}}\{\lambda(m\nu_1^{T_j}) \mid m = \nu_0^{T_j}\} \\
        &= 1 + \max\{\lambda(m\nu_1^{T_1}),\lambda(m\nu_1^{T_3})\} \\
        &=1 + \max\{\lambda(mg),\lambda(mi)\} \\
        &= 1 + \max\{6,7\}\\
        &= 8,
    \end{align*}
    For $T_1$-inner vertex $b$, 
    \begin{align*}
        \delta(b) &= 1 + \max\{\lambda(b\alpha) \mid \text{$b$ is the $T_1$-parent of $\alpha$}\} \\
        &= 1 + \max\{\lambda(ba),\lambda(bl)\}. \\
        &= 1 + \max\{2,3\} \\
        &= 4.
    \end{align*}
    We have for example, for $T_1$-inner vertex $b$ and $T_3$-inner vertex $i$, 
    \begin{align*}
        \lambda(bi) &= 6 \\
        &> \min\{\max\{2,3\},5\} \\
        &= \min\{\max\{\lambda(ba),\lambda(bl)\},\max\{\lambda(if)\}\} \\
        &= \min\{\max\{\lambda(b\alpha) \mid \text{$b$ is the $T_1$-parent of $\alpha$}\},\\
        & \ \ \ \ \ \ \ \ \ \, \max\{\lambda(i\beta) \mid i \text{ is the $T_3$-parent of $\beta$}\}\};
    \end{align*}
    for $T_1$-inner vertex $b$ and $b\nu_0^{T_3} = bm \in E$, we have
    \begin{align*}
        \lambda(bm) &= 5\\
        &> \min\{\max\{2,3\},\max\{6,7\}\} \\
        &=\min\{\max\{\lambda(ba),\lambda(bl)\},\max\{\lambda(mg),\lambda(mi)\}\} \\
        &=\min\{\max\{\lambda(ba),\lambda(bl)\},\max\{\lambda(m\nu_1^{T_1}),\lambda(m\nu_1^{T_3})\}\} \\
        &=\min\{\max\{b\alpha \mid \text{$b$ is the $T_1$-parent of $\alpha$}\},\\
        &\ \ \ \ \ \ \ \ \max_{j \in \{1,2,3,4,5\}}\{\lambda(m\nu_1^{T_j}) \mid m = \nu_0^{T_j}\}\};
    \end{align*}
    for $ms \in E$, we have
    \begin{align*}
        \lambda(ms) &= 5\\
        &> \min\{\max\{6,7\},\max\{4\}\} \\
        &=\min\{\max\{\lambda(mg),\lambda(mi)\},\max\{\lambda(se)\}\} \\
        &=\min\{\max\{\lambda(m\nu_1^{T_1}),\lambda(m\nu_1^{T_3})\},\max\{\lambda(s\nu_1^{T_2})\}\} \\
        &=\min\{\max_{j \in \{1,2,3,4,5\}}\{\lambda(m\nu_1^{T_j}) \mid m = \nu_0^{T_j}\}, \\
        &\ \ \ \ \ \ \ \ \ \ \max_{j \in \{1,2,3,4,5\}}\{\lambda(s\nu_1^{T_j}) \mid s = \nu_0^{T_j}\}\}.
    \end{align*}
    One can check that the given weighted vertex cover $(F_1 \sqcup W = \{a,\dots,y\},\delta)$ is a weighted vertex cover of $G_\lambda$ but $((F_1 \sqcup W) \smallsetminus \{v\},\delta|_{ (F_1 \sqcup W) \smallsetminus \{v\}})$ is not a weighted vertex cover of $G_\lambda$ for any $v \in F_1$. Hence there exists a minimal weighted vertex cover of $G_\lambda$ with cardinality $\geq 23$.
\end{example}

The next result says that the sufficient condition in Theorem~\ref{CMUnmixedEquivalenceThm} is also a necessary condition.

\begin{theorem} \label{CMUnmixedEquivalenceNecessaryThm}
    Let $G_\lambda$ be a weighted chordal graph. Let $F_1,\dots,F_m$ be the facets of $\Delta(G)$ which admit a free vertex satisfying that $[n]$ is the disjoint union of $F_1,\dots,F_m$. If $G_\lambda$ is Cohen-Macaulay, then $G$ and $\lambda$ satisfies the condition in Theorem~\ref{CMUnmixedEquivalenceThm}.
\end{theorem}

\begin{proof}
    Proof by contrapositive. Without loss of generality, we assume that there exists such a rooted spanning forest $\ffF$ of $G[F_1]$ consisting of components $T_1,\dots,T_k$ in which each rooted tree $T_i$ has a nonfree vertex $\nu_1^{T_i}$ as a root such that there is a nonfree vertex $\nu_0^{T_i}$ in $V \smallsetminus F_1$ with $\nu_0^{T_i}\nu_1^{T_i} \in E$ satisfying the conditions as in the statement of the theorem. Let $\llS$ be the minimal set containing $\nu_0^{T_1},\dots,\nu_0^{T_k}$. Without loss of generality, we assume $\llS \subseteq F_2 \sqcup \dots \sqcup F_i$ for some $i \in \bbN \smallsetminus \{1\}$ and $F_j \cap \llS \neq \emptyset$ for $j = 2,\dots,i$. \par
    For $j  = 2,\dots,i$, assume $F_j \cap \llS = \{\nu_0^{T_{j_1}},\dots,\nu_0^{T_{j_\ell}}\} =: V_j$ for some subset $\{j_1,\dots,j_\ell\} \subseteq \{1,\dots,k\}$, and define $\gamma_j: V_j \to \bbN$ by
    \begin{align*}
        \gamma_j(\nu_0^{T_{j_a}}) &= 1 + \max_{j' \in \{1,\dots,k\}}\{\lambda(\nu_0^{T_{j_a}}\nu_1^{T_{j'}}) \mid \nu_0^{T_{j_a}} = \nu_0^{T_{j'}}\} \\
        &= 1 + \max_{j_{a'} \in \{j_1,\dots,j_\ell\}}\{\lambda(\nu_0^{T_{j_a}}\nu_1^{T_{j_{a'}}}) \mid \nu_0^{T_{j_a}} = \nu_0^{T_{j_{a'}}}\}.
    \end{align*}
    For $j = i+1,\dots,m$, assume that $v_j \in F_j$ is a free vertex. Let
    \[W: = F_2 \sqcup \cdots \sqcup F_i \sqcup (F_{i+1} \smallsetminus \{v_{i+1}\}) \sqcup \cdots \sqcup (F_m \smallsetminus \{v_m\}).\]
    Let $\delta: F_1 \sqcup W \to \bbN$. For $j = 2,\dots,i$, set $\delta(v) = \gamma_j(v)$ for any $v \in V_j$. Set $\delta(v) = 1$ for any $v \in W \smallsetminus (V_2 \sqcup \cdots \sqcup V_i)$. If $\nu_0^{T_{j_c}} \in V_j$ and $\nu_0^{T_{j'_d}} \in V_{j'}$ such that $\nu_0^{T_{j_c}}\nu_0^{T_{j'_d}} \in E$ for some $j,j' \in \{2,\dots,i\}$, then by assumption, 
    \begin{align*}
        \lambda(\nu_0^{T_{j_c}}\nu_0^{T_{j'_d}}) &> \min\{\max\{\lambda(\nu_0^{T_{j_c}}\nu_1^Y) \mid \text{$Y$ is a component of $\ffF$ such that $\nu_0^{T_{j_c}} = \nu_0^Y$}\},\\
        &\ \ \ \ \ \ \ \ \ \, \max\{\lambda(\nu_0^{T_{j'_d}}\nu_1^Z) \mid \text{$Z$ is a component of $\ffF$ such that $\nu_0^{T_{j'_d}} = \nu_0^Z$}\}\} \\
        &= \min\{\max_{j_{c'} \in \{j_1,\dots,j_\ell\}}\{\lambda(\nu_0^{T_{j_c}}\nu_1^{T_{j_{c'}}}) \mid \nu_0^{T_{j_c}}\nu_1^{T_{j_{c'}}}\}, \\
        &\ \ \ \ \ \ \ \ \ \, \max_{j'_{d'} \in \{j_1',\dots,j_\ell'\}}\{\lambda(\nu_0^{T_{j'_d}}\nu_1^{T_{j'_{d'}}}) \mid \nu_0^{T_{j'_d}} = \nu_1^{T_{j'_{d'}}}\}\} \\
        &= \min\{\gamma_j(\nu_0^{T_{j_c}})-1,\gamma_{j'}(\nu_0^{T_{j'_d}})-1\} \\ 
        &= \min\{\delta(\nu_0^{T_{j_c}})-1,\delta(\nu_0^{T_{j'_d}})-1\},
    \end{align*}
    implying that $\min\{\delta(\nu_0^{T_{j_c}}),\delta(\nu_0^{T_{j'_d}})\} \leq \lambda(\nu_0^{T_{j_c}}\nu_0^{T_{j'_d}})$. So we have $(W,\delta|_W)$ is a weighted vertex cover of $(G - F_1)_{\lambda|_{E(G - F_1)}}$. For $j=1,\dots,k$ set $\delta(v) = 1$ for any $T_j$-leaf $v$ and set $\delta(v) = 1+\max\{\lambda(vw) \mid \text{$v$ is the $T_j$-parent of $w$}\}$ for any $T_j$-inner vertex $v$. Since $\ffF$ is a spanning forest of $G[F_1]$, we have defined $\delta(v)$ for all $v \in F_1$. If $u$ is a $T_c$-inner vertex and $v$ is a $T_d$-inner vertex for some $c,d \in \{1,\dots,k\}$ with $c \neq d$. Then by assumption,
    \begin{align*}
        \lambda(uv) &> \min\{\max\{\lambda(uw) \mid \text{$u$ is the $T_c$-parent of $w$}\},\max\{\lambda(vx) \mid \text{$v$ is the $T_d$-parent of $x$}\}\} \\
        &= \min\{\delta(u)-1,\delta(v)-1\} \\
        &= \min\{\delta(u),\delta(v)\} - 1,
    \end{align*}
    implying that $\min\{\delta(u),\delta(v)\} \leq \lambda(uv)$. For $j = 1,\dots,k$, we have 
    \begin{align*}
        \lambda(\nu_0^{T_j}\nu_1^{T_j}) &> \max\{\lambda(\nu_1^{T_j}w) \mid \text{$\nu_1^{T_j}$ is the $T_j$-parent of $w$}\} \\
        &= \delta(\nu_1^{T_j}) - 1,
    \end{align*}
    implying that $\delta(\nu_1^{T_j}) \leq \lambda(\nu_0^{T_j}\nu_1^{T_j})$. Furthermore, if $u$ is an $T_c$-inner vertex such that $u \nu_0^{T_{j_d}} \in E$ for some $c \in \{1,\dots,k\}$ and $j \in \{2,\dots,i\}$, then by assumption,
    \begin{align*}
        \lambda(u\nu_0^{T_{j_d}}) &> \min\{\max\{\lambda(uw) \mid \text{$u$ is the $T_c$-parent of $w$}\}, \\
        &\ \ \ \ \ \ \ \ \ \, \max\{\lambda(\nu_0^{T_{j_d}}\nu_1^Y) \mid \text{$Y$ is a component of $\ffF$ such that $\nu_0^{T_{j_d}} = \nu_0^Y$}\}\} \\
        &= \min\{\delta(u)-1,\max_{j_{d'} \in \{j_1,\dots,j_\ell\}}\{\lambda(\nu_0^{T_{j_d}}\nu_1^{T_{j_{d'}}}) \mid \nu_0^{T_{j_d}}\nu_1^{T_{j_{d'}}} \in E\}\} \\
        &= \min\{\delta(u)-1,\gamma_j(\nu_0^{T_{j_d}})-1\} \\
        &= \min\{\delta(u)-1,\delta(\nu_0^{T_{j_d}})-1\},
    \end{align*}
    implying that $\min\{\delta(u),\delta(\nu_0^{T_{j_d}})\} \leq \lambda(u\nu_0^{T_{j_d}})$. Hence $(F_1 \sqcup W,\delta)$ is a weighted vertex cover of $G_\lambda$. \par
    For $j = 1,\dots,k$, let $v \in F_1 \smallsetminus \{\nu_1^{T_j}\}$ and $p_v$ its $T_j$-parent, since $(F_1 \sqcup W,\delta)$ is a weighted vertex cover of $G_\lambda$ and
    \begin{align*}
        \delta(p_v) &= 1 + \max\{\lambda(p_vw) \mid \text{$p_v$ is the $T_j$-parent of $w$}\} \\
        &\geq 1 + \lambda(p_vv),
    \end{align*}
    we have $\delta(v) \leq \lambda(p_vv)$, hence $((F_1 \sqcup W) \smallsetminus \{v\}, \delta|_{(F_1 \sqcup W) \smallsetminus \{v\}})$ is no longer a weighted vertex cover of $G_\lambda$. For $j = 1,\dots,k$, there exists $j' \in \{2,\dots,i\}$ and $j'_d \in \{j'_1,\dots,j'_\ell\}$ such that $j'_d = j$, since $(F_1 \sqcup W,\delta)$ is a weighted vertex cover of $G_\lambda$ and 
    \begin{align*}
        \delta(\nu_0^{T_j}) &= \delta(\nu_0^{T_{j'_d}}) \\
        &= \gamma_{j'}(\nu_0^{T_{j'_d}}) \\
        &= 1+ \max_{j'_{d'} \in \{j'_1,\cdots,j'_\ell\}}\{\lambda(\nu_0^{T_{j'_d}}\nu_1^{T_{j'_{d'}}}) \mid \nu_0^{T_{j'_d}}\nu_1^{T_{j'_{d'}}} \in E\} \\
        &\geq 1+\lambda(\nu_0^{T_{j'_d}}\nu_1^{T_{j'_d}}),
    \end{align*}
    we have $\delta(\nu_1^{T_j}) = \delta(\nu_1^{T_{j'_d}}) \leq \lambda(\nu_0^{T_{j'_d}}\nu_1^{T_{j'_d}})$, hence $((F_1 \sqcup W) \smallsetminus \{\nu_1^{T_j}\}, \delta|_{(F_1 \sqcup W) \smallsetminus \{\nu_1^{T_j}\}})$ is not a weighted vertex cover of $G_\lambda$. So $((F_1 \sqcup W) \smallsetminus \{v\}, \delta|_{(F_1 \sqcup W) \smallsetminus \{v\}})$ is not a weighted vertex cover of $G_\lambda$ for any $v \in F_1$. Thus, there exists a minimal weighted vertex $(F_1 \sqcup W', \delta')$ of $G_\lambda$ such that $W' \subseteq W$ and $\delta' \geq \delta$ by~\cite[Proposition 1.12]{MR3055580}. Since $F_2,\dots,F_m$ are cliques of $G$, we have $\abs{W'} \geq \abs{W} - (i-1)$. So
    \[\abs{(F_1 \sqcup W', \delta')} \geq \abs{F_1} + \abs{W} - (i-1) = n - (m-i) - (i-1) = n-m+1.\]
\par On the other hand, assume that $v_j \in F_j$ is free for $j = 1,\dots,m$. Let $U = V \smallsetminus \{v_1,\dots,v_m\}$ and $\delta': U \to \bbN$ defined by $\delta'(v) = 1$. Then $(U,\delta')$ is a weighted vertex cover of $G_\lambda$. For any $v \in U$, we have $v \in F_j$ for some $j \in \{1,\dots,m\}$ and so $vv_j \in E$, hence $(U \smallsetminus \{v\},\delta'|_{U \smallsetminus \{v\}})$ is not a weighted vertex cover of $G_\lambda$. So there is a minimal weighted vertex cover $(U,\delta'')$ of $G_\lambda$ such that $\delta'' \geq \delta'$ by~\cite[Proposition 1.12]{MR3055580}, but $\abs{(U,\delta'')} = \abs {U} = n - m$. Therefore, $G_\lambda$ is mixed and so $G_\lambda$ is not Cohen-Macaulay by Theorem~\ref{CMUnmixedEquivalenceThm}. 
\end{proof}

\section*{Acknowledgments}
I am grateful for the insightful comments and feedback provided by Keri Sather-Wagstaff and Janet Vassilev. 

\bibliographystyle{plain} 
\bibliography{bibliography}

\end{document}